\documentclass[12pt]{amsart}
\usepackage{amssymb,amscd}
\usepackage{pb-diagram}
\usepackage{graphicx}

\input epsf
\newtheorem{thm}{Theorem}[section]
\newtheorem{lem}[thm]{Lemma}

\newtheorem{prop}[thm]{Proposition}
\newtheorem{rem}[thm]{Remark}

\newtheorem*{theorem*}{Theorem}
\newtheorem*{corollary*}{Corollary}
\newtheorem*{proposition*}{Proposition}
\newtheorem*{theoremA*}{Theorem A}
\newtheorem*{corollaryB*}{Corollary B}
\newtheorem*{propositionC*}{Proposition C}
\newtheorem*{propositionD*}{Proposition D}

\numberwithin{equation}{section}
\begin{document}

\thanks{
Goldman gratefully acknowledges partiall support by
National Science Foundation grant DMS070781.
Xia gratefully acknowledges partial support by the National Science Council, Taiwan with grants 96-2115-M-006-002 and 97-2115-M-006-001-MY3.}

\title[Action of the Johnson-Torelli group]
{Action of the Johnson-Torelli group on representation varieties}

\author{William M. Goldman}
\address{
Department of Mathematics,\\
University of Maryland,\\
College Park, MD 20742 \\
\tt{wmg@math.umd.edu} ({\it Goldman}) }

\author{Eugene Z. Xia}
\address{
National Center for Theoretical Sciences\\
Department of Mathematics\\
National Cheng-kung University\\
Tainan 701, Taiwan \\
\tt{ezxia@ncku.edu.tw}
({\it Xia})}

\date{\today}

\newcommand{\B}{{\mathbb B}}
\newcommand{\R}{{\mathbb R}}
\newcommand{\Z}{{\mathbb Z}}
\newcommand{\Q}{{\mathbb Q}}
\newcommand{\C}{{\mathbb C}}

\newcommand{\QQ}{{\mathcal Q}}

\newcommand{\la}{\langle}
\newcommand{\ra}{\rangle}
\renewcommand{\Im}{\mathsf{Im}}
\renewcommand{\Re}{\mathsf{Re}}
\newcommand{\tr}{\mathsf{tr}}
\newcommand{\Hom}{\mathsf{Hom}}
\newcommand{\Ham}{\mathsf{H}}
\newcommand{\Mod}{\mathsf{MCG}}
\newcommand{\Aut}{\mathsf{Aut}}
\newcommand{\Out}{\mathsf{Out}}
\newcommand{\Inn}{\mathsf{Inn}}
\newcommand{\Ker}{\mathsf{Ker}}
\newcommand{\Ad}{\mathsf{Ad}}
\newcommand{\Id}{\mathbb{I}}

\newcommand{\Fm}{{\mathbb F}_m}
\newcommand{\Ft}{{\mathbb F}_3}
\newcommand{\su}{{\mathfrak{su}(2)}}
\newcommand{\SU}{{\mathsf{SU}(2)}}
\newcommand{\slt}{\mathsf{SL}(2,\C)}
\newcommand{\Un}{\o{U}(n)}
\newcommand{\Uo}{\mathsf{U}(1)}

\newcommand{\Mgn}{{\Sigma_{g,n}}}
\newcommand{\MC}{{\Hom_\Cc(\pi,\SU)/\SU}}
\newcommand{\XS}{{\Hom(\pi,\slt)//\slt}}
\renewcommand{\P}{{\mathbb P}}
\newcommand{\bH}{{\mathbb H}}
\newcommand{\M}{{\mathcal M}}
\newcommand{\m}{{\mathfrak M}}
\newcommand{\bM}{{\bar M}}
\newcommand{\Cc}{{\mathcal C}}
\newcommand{\Ff}{{\mathcal F}}
\newcommand{\ff}{{\mathfrak F}}
\newcommand{\f}{\mathsf{t}}
\newcommand{\ttt}{\mathsf{t}}
\newcommand{\F}{\mathsf{F}}
\newcommand{\p}{\mathsf{p}}
\newcommand{\BB}{\mathcal{B}}

\newcommand{\Gg}{\mathfrak{G}}

\newcommand{\Tt}{\mathcal{T}}
\newcommand{\sS}{\mathcal{S}}
\newcommand{\Jj}{\mathcal{J}}
\newcommand{\Xc}{\X_\Cc}
\newcommand{\X}{\mathfrak{X}}
\newcommand{\XcU}{\X^U_\Cc}
\renewcommand{\c}{\mathsf{c}}
\newcommand{\G}{{\mathcal G}}

\newcommand{\hpg}{\Hom(\pi,G)/G}
\newcommand{\hppg}{\Hom(\pi,G)\slashslash G}
\newcommand{\hmg}{\Hom(\pi,G)}
\renewcommand{\k}{{\mathfrak k}}
\renewcommand{\H}{\text{H}}

\begin{abstract}
Let $\Sigma$ be a compact orientable surface with genus $g$ and
$n$ boundary components $B = (B_1,\dots
,B_n)$. Let $c = (c_1,\dots,c_n)\in [-2,2]^n$. Then the
mapping class group $\Mod$ of $\Sigma$ acts on the relative
$\SU$-character variety $\Xc:=\MC$, comprising conjugacy classes of
representations $\rho$ with $\tr(\rho(B_i)) = c_i$. This
action preserves a symplectic structure on the smooth part of $\Xc$,
and the corresponding measure is finite. Suppose $g =1$ and $n = 2$.
Let $\Jj \subset \Mod$ be the subgroup generated by Dehn twists
along null homologous simple loops in $\Sigma$. Then the action of $\Jj$ on
$\Xc$ is ergodic for almost all $c$.
\end{abstract}
\maketitle

\section{Introduction}

Let $\Sigma = \Mgn$ be a compact oriented surface of genus $g$
with $n$ boundary components
$B = \{B_1,...,B_n\}.$
Let $\pi = \pi_1(\Sigma)$ denote its fundamental group.
The {\em mapping class group\/} $\Mod$ consists of
isotopy classes of orientation-preserving homeomorphisms of $\Sigma$
which pointwise fix each $B_i$.
Alternatively, $\Mod$ is the image under the quotient
homomorphism
\begin{equation*}
\Aut(\pi) \longrightarrow \Out(\pi) := \Aut(\pi)/\Inn(\pi)
\end{equation*}
of the subgroup $\Aut(\pi,\BB)$ of all automorphisms of $\pi$ that preserve the set $\BB$ of  conjugacy classes
of the cyclic subgroups $\pi_1(B_i)\subset\pi$ and correspond
to orientation-preserving homeomorphisms.

Let \(G\) be a Lie group.
Then $G$ acts on $\Hom(\pi,G)$ by conjugation.  Let
$$
\X(G) = \Hom(\pi,G)/G.
$$
Let $\Cc = \{C_1, \cdots, C_n\}$, where $C_i \subseteq G$ is a conjugacy class for $1 \le i \le n$.  Then the relative
representation variety is
$$
\Hom_\Cc(\pi, G) = \{\rho \in \Hom(\pi, G): \rho(B_j) \in
C_j, \ \  \mbox{for} \ \  1 \le j \le n \}.
$$
The group $G$ acts on $\Hom_\Cc(\pi,G)$ by conjugation and the
moduli space is the quotient
$$
\Xc(G) = \Hom_\Cc(\pi, G)/G.
$$

The group $\Aut(\pi,\BB)$ acts on $\pi$, preserving $\BB$.
Hence it acts on $\Hom_\Cc(\pi,G)$.
Furthermore the action descends to a $\Mod$-action on $\Xc(G)$. The
moduli space $\Xc(G)$ has an invariant dense open subset $\XcU(G)$ which
is a smooth manifold. This subset has an $\Mod$-invariant
symplectic structure $\omega$, hence, a natural smooth $\Mod$-invariant measure
$\mu$ \cite{Nature, Hu1}.

Denote by $S$ the set of homotopy classes of simple closed curves on $\Sigma$ and by $J \subseteq S$ the null homologous (in $\H_1(\Sigma,\Z)$) subset.
The group $\Mod$ is generated by Dehn twists $\tau_a$ along simple
loops in $S$.  Denote by $\Jj \subseteq \Mod$ the subgroup generated
by Dehn twists along simple loops in $J$ and by $\Tt \subseteq \Mod$
the subgroup generated by Dehn twists $\tau_a$ for $a \in J$ and
products $\tau_a \tau_b^{-1}$, where $a$ and $b$ are disjoint but homologous
simple loops in $S$.

When $n \le 1$, $\Tt$ is the {\em Torelli group}, i.e. the subgroup of $\Mod$ acting trivially on $\H^1(\Sigma,\Z)$
\cite{Johnson1}.   Johnson constructed epimorphisms
\begin{equation*}
\left\{
\begin{array}{ll}
\Tt \longrightarrow \Lambda^3 \H^1(\Sigma,\Z)/\H^1(\Sigma, \Z)   & \text{ for } n = 0,\\
\Tt \longrightarrow \Lambda^3 \H^1(\Sigma,\Z) & \text{ for } n = 1
\end{array}
\right.
\end{equation*}
and define the kernels to be $\Jj$~\cite{Johnson2, Johnson3}.

For $n > 1$, our definition of $\Tt$ relates to the functorial Torelli group (see \cite{Pu1, VDB}).
The ergodicity of the $\Mod$-action on $\Xc(\SU)$ was proved in \cite{Erg, GX1}.  See \cite{PX1, PX2} for similar results when $G$ is a general compact group.
Here we prove the following ergodicity result:
\begin{thm}\label{thm:main}
Suppose $g = 1$ and $n = 2$.  Then the $\Jj$-action on $\Xc(\SU)$ is ergodic for generic $C_1$ and $C_2$.
\end{thm}

The moduli space $\Xc(\SU)$ possesses a symplectic structure.  The group $\Jj$ is generated by simple loops described above.  The same simple loops also correspond to fundamental group elements.  These Dehn twist actions embed into the Hamiltonian vector field flows of the trace functions on these corresponding fundamental group elements.  It is then a routine matter to produce a set of such Hamiltonian vector fields, whose flows are locally transitive on an open dense (Zariski) subset $U \subseteq \Xc(\SU)$.  However since $\Xc(\SU)$ is a real variety ($\SU < \slt$ is a real form), $V = \Xc(\SU) \setminus U$ is of $\R$-codimension 1.  In other words, $V$ may contain ``walls'' between components of $U$.  To prove ergodicity of the $\Jj$-action, we analyze the vector fields along $V$ explicitly.  This then requires an explicit computation of the symplectic form with the aid of a computer.  The inability to carry out these explicit computations for curves of higher genuses and/or with more punctures is the main obstacle in generalizing Theorem~\ref{thm:main} to these curves.

\section{Trace functions and Hamiltonian flows}

This section summarizes some needed results from \cite{GX1}.
Let $X$ be a symplectic manifold and $f : X \to \R$ a smooth function.  Denote by $\Ham(f)$ the associated Hamiltonian vector field.

\begin{prop}\label{prop:Ham}
Let $X$ be a connected symplectic manifold and let $\Ff$ be a set of
real smooth $\R$-valued functions on $X$ such that at every point $x\in X$, the
differentials $df(x)$, for $f\in\Ff$, span the cotangent space
$T_x^*(X)$. Then the group generated by the Hamiltonian flows
of the vector fields $\Ham(f)$, for $f\in\Ff$, acts
transitively on $X$.
\end{prop}
\begin{proof}
The proof is a straightforward application of the implicit function therem; see Lemma 3.2 in \cite{GX1}.
\end{proof}

Let $G = \slt$ and $\Cc = \{C_1, \cdots , C_n\}$ be a family of conjugacy classes in $G$ such that $C_i$ is non-parabolic for each $1 \le i \le n$.  Let $c = (c_1,\cdots,c_n) \in \C^n$ such that
$c_i = \tr(A) \in \C$ for all $A \in C_i$.
Then the representation variety is equivalently defined as
\begin{align*}
\Hom_\Cc(\pi, G) &= \{\rho \in \Hom(\pi, G): \ \ \tr(\rho(B_j)) & = c_j, \ \  \mbox{for} \ \  1 \le j \le n \}.
\end{align*}

In this setting, if $\alpha\in\pi$ is a homotopy class of based loops,
then $\f_\alpha$, the {\em trace function\/} of $\alpha$ on $\Xc$, is defined as
$$
\Hom(\pi,\slt) \xrightarrow{\f_\alpha} \C; \ \
\rho \longmapsto  \tr\big(\rho(\alpha)\big).
$$
Since the function $\slt \xrightarrow{\tr} \C$ is $\Inn(\pi)$-invariant,
$\f_\alpha$ defines a function (also denoted by $\f_\alpha$) on $\Xc(\slt)$.
Furthermore, when $G = \SU$, $c \in I^n$, where $I = [-2,2]$.



\begin{prop}\label{prop:discreteflow}
Let $\alpha$ be a simple separating curve on $\Sigma$
with
Dehn twist
$\tau_\alpha$. Let $\psi : \Xc \to \R$ be a
measurable function invariant under the cyclic group
$\langle (\tau_\alpha)_* \rangle$. Then $\psi$ is almost everywhere invariant
under the flow of $\Ham(\f_\alpha)$.
\end{prop}
\begin{proof}
See Proposition 5.4 in \cite{GX1}.

\end{proof}
For the rest of the paper, we shorten $\Xc(\SU)$ (resp. $\XcU(\SU)$) to $\X_c$ (resp. $\X_c^U$).

\section{Ergodicity}
For $g = 0 \ \text{ and } \ n = 4$ or for $g = 1 \ \text{ and } \ n = 2$, the fundamental group $\pi$ is isomorphic to the free group of three generators
$$
\Ft = \langle F_1, F_2, F_3, F_4 \mid \prod F_i \rangle,
$$
where $F_i$ corresponds to a simple closed curve on $\Sigma$.
By convention, we also use elements in $\pi$ to denote curves they represent
on $\Sigma$.


The coordinate ring of $\Hom(\Ft,\slt)/\slt$ is $\C[K]/(k)$, where
$$
K = \{\f_4, \f_1,\f_2,\f_3,\f_{12},\f_{13},\f_{23}\}
$$
with
$$
\f_{i} = \tr(\rho(F_i)), \ \ \f_{ij} = \tr(\rho(F_i F_j)) \text{ for
} \rho \in \Hom(\Ft, \slt)
$$
and
\begin{equation}\label{eq:fourhole}
k := \f_{12}^2 + \f_{23}^2 + \f_{13}^2 + \f_{12}\f_{23}\f_{13} -
\c_{12} \f_{12} - \c_{23} \f_{23} - \c_{13} \f_{13}  - \c_0,
\end{equation}
where
\begin{align}\label{eq:fourholecoeffs}
\c_{12} & :=  \f_1 \f_2 + \f_3 \f_4, \notag\\
\c_{23} & :=  \f_2 \f_3 + \f_1 \f_4, \notag\\
\c_{13} & :=  \f_1 \f_3 + \f_2 \f_4, \notag \\
\c_0 & :=  4 - \f_1^2 - \f_2^2 - \f_3^2 - \f_4^2 - \f_1 \f_2 \f_3
\f_4.
\end{align}
Let
$$\Omega = d\f_4 \wedge d\f_1 \wedge d\f_2 \wedge d\f_3 \wedge d\f_{12} \wedge d\f_{13} \wedge d\f_{23}.$$
A priori, we consider elements in $\R[K]$ as functions on $I^7$.  Then $\X \subseteq I^7$ is a compact component of
$\{v \in I^7 : k(v) = 0\}.$
Hence elements in $\R[K]$ also restrict to functions on $\X$.

\subsection{The 4-holed sphere}
Suppose $g=0$ and $n=4$.
The boundary components of $\Sigma$ are
$$B = (B_1, B_2, B_3, B_4).$$  The fundamental group $\pi$ is isomorphic to $\Ft$ with the isomorphism $B_i
\mapsto F_i$.
Let
$$
\F : \X \longrightarrow I^4, \ \ \F(K) = (\f_1, \f_2, \f_3, \f_4).
$$
Then for $c = (c_1, c_2, c_3, c_4) \in I^4$, $\X_c$ is a compact component of $\F^{-1}(c)$.

The Johnson kernel $\Jj$ for the 4-holed torus is trivial as any non-trivial Dehn twist must be along simple curves that separate the four boundary components into pairs.  However one may study a different group action as follows:  Fix the boundary components into two pairs $\{B_1, B_2\}$ and $\{B_3, B_4\}$.  Let $J' \subseteq S$ be the subset containing all the curves separating $\Sigma$ into two pairs of pants containing $\{B_1, B_2\}$ and $\{B_3, B_4\}$, respectively.  Let $\Jj' \subseteq \Mod$ be the subgroup generated by Dehn twists along elements in $J'$.  In this section, we study the $\Jj'$-action on $\X_c^U$.  This problem is interesting in its own right and instructive in the study of the $\Jj$-action on $\X_c^U$ when $\Sigma$ is the 2-holed torus.

The symplectic bi-vector field relating to $\omega$ is $$W =
(\partial_{\f_{12}} \wedge
\partial_{\f_{23}} \wedge
\partial_{\f_{13}})(d k).$$


By convention, we use $F_i$ to also denote a simple closed curve it represents.  The Dehn twist along the simple closed curve $F_2 F_3$ takes the simple
closed curve $F_1F_2$ to a simple closed curve $F_0$. Denote by
$\tau_{12}, \tau_0$ the Dehn twists along $F_1 F_2$ and $F_0$,
respectively.  Let
$$\Gamma = \langle \tau_{12}, \tau_0\rangle.$$
\begin{rem} \label{rem:T}
Both $F_1 F_2$ and $F_0$ separate $\Sigma$ into two components each
containing $\{B_1, B_2\}$ and $\{B_3, B_4\}$, respectively. Hence
$\Gamma \subseteq \Jj'$.
\end{rem}
Let $\M$ be the space
of measurable functions $\X \to \R$.  The trace functions of $F_1F_2, F_0$ are, respectively,
$$\p_{12} = \f_{12}, \ \ \ \p_0 = \c_{12} - \f_{23} \f_{13} - \f_{12}.
$$
Let $\Ham_i = W(d \p_i)$ be the Hamiltonian vector field (notice that, to conserve notation, the subscript index $i$ may mean either a number or a pair of numbers).
Let $\Gg_i = \Gg(\p_i)$ where $\Gg(\p_i)$ is the group generated by the Hamiltonion flow of $\Ham(\p_i)$.
Let $\Gg$ be the group generated by $\bigcup_i \Gg_i$.
Suppose $f \in \M^{\langle \tau_i\rangle}$. Then by Proposition~\ref{prop:discreteflow},
$f \in \M^{\Gg_i}$. Hence $\M^{\langle \tau_i\rangle} \subseteq \M^{\Gg_i}.$ Therefore
$$
\M^\Gamma \subseteq \bigcap_i \M^{\langle \tau_i\rangle} \subseteq
\bigcap_i \M^{\Gg_i} = \M^\Gg.
$$

Let  $s \in \R[K]$ such that
$$s \Omega = d\f_4 \wedge d\f_1 \wedge d\f_2 \wedge d\f_3 \wedge dk \wedge d\p_0 \wedge d\p_{12}.$$
A direct calculation shows that
$$
s = 2\f_{23}^2- \c_{23} \f_{23} - 2\f_{13}^2 + \c_{13} \f_{13}.
$$
Then
$$
V = \{v \in I^7 : s(v) = 0\}
$$
is the {\em dependency locus}; i.e. by definition,
\begin{lem}\label{lem:singular}
The Hamiltonian vector fields $\Ham_0, \Ham_{12}$ are linearly dependent at
$v$ only if $v \in V$.
\end{lem}
\begin{rem}
The purpose is to find a sufficient condition for the set $\{d\p_i\}$ to be linearly independent at $v \in \X_c$.  The choices of $d\f_i, \ 1 \le i \le 4$ and $dk$ are not canonical.  They are chosen because the $f_i \ 's \ (1 \le i \le 4)$ are the boundary traces and $k$ is the defining equation of $\X$; hence, they are constant on $\X_c$.
\end{rem}

\begin{lem}\label{lem:transverse}
The vector field $\Ham_{12}$ is not tangent to $V$ for almost all $c \in I^4$.
\end{lem}
\begin{proof}
Suppose $\Ham_{12}$ is tangent to $V$.  Then $\Ham_{12}$ is a derivation on the ring of functions on $V \cap \X$.
Hence
$$
\Ham_{12}(s) =
  (2 \f_{13} + \f_{12} \f_{23} - \c_{13})
  (4 \f_{23} - \c_{23}) + (-2 \f_{23} - \f_{12} \f_{13} + \c_{23})
         (-4 \f_{13} + \c_{13})
$$
is zero in $\R[K]/(k,s)$; that is,  $\Ham_{12}(s) \in (k,s)$.

We now compute a Gr\"obner basis for $(k,s) \subseteq \R[K]$.
A direct computation shows that the residue of $\Ham_{12}(s)$ is not zero.  This implies that $\Ham_{12}(s) \not\in (k,s)$.  This implies that for a generic $c \in I^4$, $\Ham_{12}$ is not tangent to $V$.
\end{proof}
\begin{thm} \label{thm:4-hole}
The $\Gamma$-action on $\X_c$ is ergodic for almost every $c \in I^4$.
\end{thm}
\begin{proof}  Suppose $f \in \M^\Gamma$.  Then $f \in \M^\Gg$.
For almost all $c \in I^4$, the
set $Q=\X_c^U \cap V$ has measure zero and divides $\X_c^U$ into a
finite number of components. Let $A \subseteq \X_c^U \setminus Q$ be a
connected component. By Lemma~\ref{lem:singular}, the fibres of
$\p_{12}, \p_0$ are tangent to each other at $v$ only if $v \in Q$.
Hence, by Proposition \ref{prop:Ham}, $f$ must be constant almost everywhere on $A$.

Lemma~\ref{lem:transverse} implies that there is a Zariski dense
subset of $Q$ upon which $\Ham_{12}$ is not tangent to $Q$. Hence there
exists a smooth vector field (namely $\Ham_{12}$) that flows across
$Q$ between adjacent components.  This implies that if $v_0, v_1 \in
\X_c^U \setminus Q$, then there exists $g \in \G$ such that $g(v_0) =
v_1$.  Since $\X_c^U$ is smooth and connected, $f$ is constant almost everywhere on $\X_c^U \setminus Q$.
Since $\X_c^U$ is open and dense in $\X_c$ and $Q$ has measure zero, the theorem follows.
\end{proof}

\subsection{The 2-holed torus}  This case runs in parallel with that of the 4-holed sphere above.  Let $g = 1$ and $n = 2$.  The boundary components of $\Sigma$ are
$B = (B_1, B_2)$.  The fundamental group $\pi$ is isomorphic to $\Ft$ as described by the ribbon graph in Figure~\ref{fig:twoholedtorus}, with $B_1 \mapsto F_1 F_2 F_3$ and $B_2 \mapsto F_1 F_3 F_2$.

\begin{figure}[h]
\centerline{\epsfxsize=3.5in \epsfbox{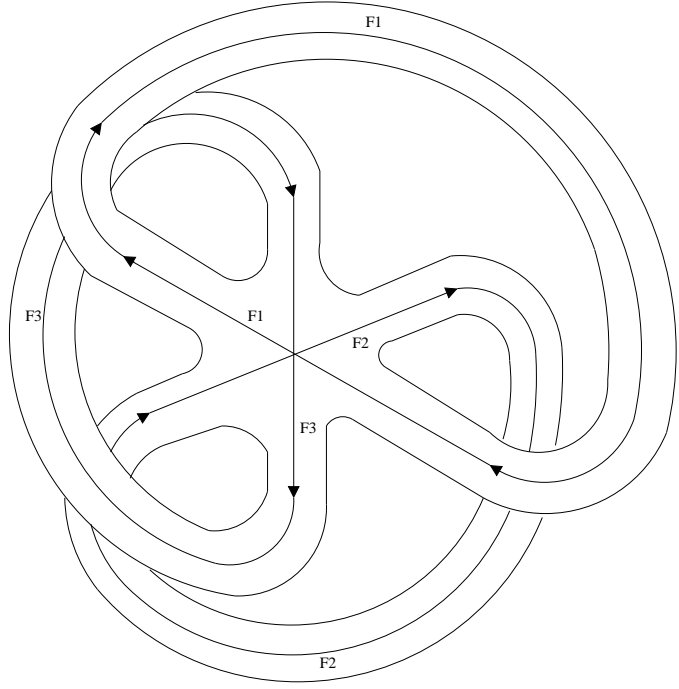}}
\caption{A ribbon graph representing a 2-holed torus}
\label{fig:twoholedtorus}
\end{figure}



The sum and product formulas are, respectively \cite{Trace}:
$$
k_s = \f_{13} \f_{2} + \f_1 \f_{23} + \f_{12} \f_3 - \f_1 \f_2 \f_3,
$$
$$
k_p = \f_2^2 + \f_3^2 + \f_1^2 + \f_{12}^2 + \f_{13}^2 + \f_{23}^2 + \f_{12} \f_{13} \f_{23} - \f_2 \f_3 \f_{23} -
  \f_2 \f_1 \f_{12} - \f_3 \f_1 \f_{13} - 4.
$$
Notice that $k = k_p - \f_4(k_s - \f_4).$  Let
$$
\F : \X \longrightarrow I^2, \ \ \F(K) = (\f_4, k_s - \f_4).
$$
Then for $c = (c_1, c_2) \in I^2$, $\X_c$ is a compact component of $\F^{-1}(c).$
\begin{rem}
The moduli space $\X_c$ is the subspace of $\X$ defined by $k_s = c_1 + c_2$ and $k_p = c_1 c_2$.
\end{rem}

With respect to the coordinates $K$, the
canonical
symplectic bi-vector field $W$ relating to $\omega$ is (this can be computed from the ribbon diagram of Figure~\ref{fig:twoholedtorus}) the $7 \times 7$ skew-symmetric matrix and specified by the following (The unspecified terms are zero; i.e. after filling in these specified terms and making it skew-symmetric, the rest of the terms are zero):
$$
\left\{
\begin{array}{l}
W_{2,3} = -2\f_{12} + \f_1 \f_2, \ \ W_{2,4} = 2 \f_{13} - \f_1 \f_3, \\
W_{2,5} = -\f_1 \f_{12} + 2 \f_2, \ \  W_{2,6} = \f_1 \f_{13} - 2 \f_3, \\
W_{3,4} = -2 \f_{23} + \f_2 \f_3, \ \ W_{3,5} = -2 \f_1 + \f_{12} \f_2, \ \ W_{3,7} = -\f_2 \f_{23} + 2 \f_3, \\
W_{4,6} = 2 \f_1 - \f_{13} \f_3, \ \ W_{4,7} = -2\f_2 + \f_{23} \f_3, \\
W_{5,6} = -\f_{12}\f_{13} - 2\f_{23} + 2\f_2\f_3, \ \ W_{5,7} = 2 \f_{13} + \f_{12}\f_{23} -2\f_1\f_3,\\
W_{6,7} = -2\f_{12} + 2\f_1\f_2 - \f_{13}\f_{23}.
\end{array}
\right.
$$
The Dehn
twist along the simple closed curve $F_1$ takes the simple closed
curve $F_2 F_3 F_2^{-1} F_3^{-1}$ to a simple closed curve $F_0$.
Denote by $\tau_{ij}$ the Dehn twists along $F_i F_j F_i^{-1}
F_j^{-1}$ and by $\tau_0$ the Dehn twist along $F_0$.  Let
$$\Gamma = \langle \tau_{12}, \tau_{23}, \tau_{13}, \tau_0 \rangle.$$
\begin{rem} \label{rem:J}
$\Gamma \subseteq \Jj$.
\end{rem}

Let $\M$ be the space
of all measurable functions $\X \to \R$.
The trace functions for various $i$ are
$$
\p_{12} = -2 + \f_1^2 + \f_{12}^2 - \f_1\f_{12}\f_2 + \f_2^2,
$$
$$
\p_{23} = -2 + \f_2^2 + \f_{23}^2 - \f_2\f_{23}\f_3 + \f_3^2,
$$
$$
\p_{13} = -2 + \f_1^2 + \f_{13}^2 - \f_1\f_{13}\f_3 + \f_3^2,
$$
$$
\p_0 = \c_0 + \f_1\f_{12}\f_2 + \f_4\f_{13}\f_2 +
\f_1 \f_4 \f_{23}
$$
$$
+ \f_{12} \f_4 \f_3 - \f_1\f_{13}\f_3  -
\f_1\f_{12}\f_{23}\f_3 + \f_2\f_{23}\f_3 + \f_1^2\f_3^2 - 2.
$$

Let $\Ham_i = W(d \p_i)$
(again the subscript $i$ may denote a number or a pair of numbers)
be the Hamiltonian vector field.
Let $\Gg_i = \Gg(\p_i)$ where $\Gg(\p_i)$ is the group generated by the Hamiltonion flow of $\Ham(\p_i)$.
Let $\Gg$ be the group generated by $\bigcup_i \Gg_i$.
Suppose $f \in \M^{\langle \tau_i\rangle}$. Then by Proposition~\ref{prop:discreteflow},
$f \in \M^{\Gg_i}$. Hence $\M^{\langle \tau_i\rangle} \subseteq \M^{\Gg_i}.$ Therefore
$$
\M^\Gamma \subseteq \bigcap_i \M^{\langle \tau_i\rangle} \subseteq
\bigcap_i \M^{\Gg_i} = \M^\Gg.
$$

Let $V \subseteq I^7$ be the subvariety defined by the polynomial $s = 0$, where
$$s \Omega = d \f_4 \wedge d (k_p - \f_4) \wedge d k \bigwedge_i d \p_i.$$
By definition, $V$ is the dependency locus:
\begin{lem}\label{lem:2hole singular}
If the set $\{\Ham_i\}$ is linearly dependent at $v \in \X_c^U$, then $v \in V$.
\end{lem}
\begin{rem}
Again in complete parallel with the case of the 4-hole sphere, the choices of $d\f_4, d(k_s - \f_4), dk$ are not canonical.  They are chosen because $\f_4, k_s - \f_4$ are the boundary traces and $k$ is the defining equation of $\X$; hence, they are constant on $\X_c$.
\end{rem}
\begin{lem}\label{lem:2hole transverse}
The vector field $\Ham_{12}$ is not tangent to $V$ for almost all $c \in I^2$.
\end{lem}
\begin{proof}
Suppose $\Ham_{12}$ is tangent to $V$.  Then $\Ham_{12}$ is a derivation
on the ring of functions on $V \cap \X$. Hence $\Ham_{12}(s) = 0 \in  \R[K]/(k,s).$  This implies that $H_{12}(s) \in (k,s)$.

With the aid of a computer, one may compute a Gr\"obner basis for $(k,s)$ 
and show that the residue of $\Ham_{12}(s)$ is not zero.  Hence $\Ham_{12}(s) \not\in (k,s)$.  This implies that for a generic $c \in I^2$, $\Ham_{12}$ is not tangent to $s$.

It so happens that $s = s_1 s_2$ is reducible with two factors.  Hence one may compute the Gr\"obner basis $(k, s_1)$ and $(k, s_2)$ and then compute the residues in each cases.

\end{proof}

\begin{prop} \label{thm:2-hole}
The $\Gamma$-action on $\X_c^U$ is ergodic for almost every $c \in I^2$.
\end{prop}
\begin{proof}
Suppose
$f \in \M^\Gamma$.  Then $f\in \M^\Gg$.
For almost all $c \in I^2$, the set $Q=\X_c^U \cap
V$ has measure zero and divides $\X_c^U$ into a finite number of
components. Let $A \subseteq \X_c^U \setminus Q$ be a connected
component. By Lemma~\ref{lem:2hole singular} and Proposition \ref{prop:Ham},
$f$ is constant almost everywhere on $A$.

Lemma~\ref{lem:2hole transverse} implies that there is a Zariski dense
subset of $Q$ upon which $\Ham_{12}$ is not tangent to $Q$. Hence there
exists a smooth vector field (namely $\Ham_{12}$) in $\G_{12}$ that flows
across $Q$ between adjacent components.  This implies that if $v_0,
v_1 \in \X_c^U \setminus Q$, then there exists $g \in \G$ such that
$g(v_0) = v_1$.  Since $\X_c^U$ is smooth and connected, $f$ is constant almost everywhere on $\X_c^U \setminus Q$.
Since $Q$ has measure zero, the theorem follows.
\end{proof}

Theorem~\ref{thm:main} follows as $\Gamma \subseteq \Jj$ and $\X_c^U$ is open and dense in $\X_c$.

\end{document}